\documentclass[10pt,oneside]{amsart}
\usepackage{amssymb, amscd, amsmath, amsthm, latexsym, hyperref}
\usepackage{pb-diagram, fancyhdr, graphicx, psfrag, enumerate}
\usepackage[text={6.3in,8.5in},centering,letterpaper,dvips]{geometry}
\newcommand{\compactlist}{\begin{list}{$\bullet$}{\setlength{\leftmargin}{1em}}}

\def\zz{{\mathbb Z}}

\def\qq{{\mathbb Q}}
\def\cc{{\mathbb C}}
\def\ss{{\mathbb S}^1}
\def\rr{{\mathbb R}}

\def\cs{\mathbin{\#}}

\newcommand{\spinc}{\ifmmode{{\mathfrak s}}\else{${\mathfrak s}$\ }\fi}
\newcommand{\spinct}{\ifmmode{{\mathfrak t}}\else{${\mathfrak t}$\ }\fi}

\newtheorem{theorem}{Theorem}
\newtheorem{lemma}{Lemma}

\theoremstyle{definition}

\begin{document}
\title{Signature functions of knots}
\author{Charles Livingston}
\thanks{This work was supported by a grant from the National Science Foundation, NSF-DMS-1505586.}
\address{Charles Livingston: Department of Mathematics, Indiana University, Bloomington, IN 47405 }
\email{livingst@indiana.edu}


\begin{abstract}    The signature function of a knot is an integer-valued step function on the unit circle in the complex plane.  Necessary and sufficient conditions for   a function   to be the  signature function of a knot are presented.  
\end{abstract}

\maketitle


\section{Introduction}
 
For a knot $K \subset S^3$, the signature function, $\sigma_K(\omega)$, is an integer-valued step function defined on the unit circle $\ss \subset \cc$. Its discontinuities can occur only at roots of the Alexander polynomial, $\Delta_K(t)\in \zz[t,t^{-1}]$. The function is {\it balanced}, in the sense that for all $t \in \rr$,
 $$\sigma_K(e^{2 \pi i t}) = \frac{1}{2}\lim_{\epsilon\to 0^+} \left( \sigma_K(e^{  2 \pi i (t+\epsilon)}) + \sigma_K(e^{ 2 \pi i (t-\epsilon)} )\right).$$
There is an associated {\it jump} function, $$J_K(e^{2 \pi i t}) = \frac{1}{2}\lim_{\epsilon\to 0^+} \left( \sigma_K(e^{  2 \pi i (t+\epsilon)}) - \sigma_K(e^{ 2 \pi i (t-\epsilon)} )\right).$$

Seifert~\cite{seifert} (see also~\cite{levine2}) characterized the set of polynomials that occur as Alexander polynomials of knots:  if $\Delta(t) \in \zz[t,t^{-1}]$, there exists a knot $K$ such that $\Delta_K(t) = \Delta(t)$ if and only if $\Delta(1) = \pm 1$ and $\Delta(t) =t^k \Delta(t^{-1})$ for some $k \in \zz$.   In general, the Alexander polynomial is well-defined up to multiplication by $\pm t^k$. We refer to any integer polynomial that satisfies these conditions as an {\it Alexander polynomial}. 
 
Here we characterize the set of signature functions of knots.  Recall that two   complex numbers are called {\it Galois conjugate} if they are roots of the same rational irreducible polynomial.
 
 \begin{theorem}\label{thm:main} Let $\sigma$ be a balanced integer-valued step function on $\ss \in \cc$.  Then $\sigma = \sigma_K$ for some knot $K$ if and only if:
 \begin{enumerate}
 \item $\sigma(\omega) = \sigma(\overline{\omega})$ for all $\omega \in \ss$.\vskip.05in
 \item $\sigma(1) = 0$.\vskip.05in
 \item All discontinuities of $\sigma$ occur at roots of Alexander polynomials.\vskip.05in
 \item If $\alpha_1 \in \ss$ and $\alpha_2 \in \ss$ are Galois conjugate, then $J(\alpha_1) \equiv J(\alpha_2) \mod 2$.
 \end{enumerate}
 
 \end{theorem}
 
Before proceeding to the proof, we briefly present  background.  The signature of a knot, now viewed as $\sigma_K(-1)$, was first defined by Trotter~\cite{trotter}  and Murasugi~\cite{murasugi}.  The signature function is essentially due to Levine~\cite{levine} and Tristram~\cite{tristram}.  Milnor~\cite{milnor} defined a different set of invariants, now called {\it Milnor signatures}, and these were proved be equivalent to the jumps in the signature function by Matumoto~\cite{matumoto}. 
 
One can define the knot signature function as we do below, but without taking the two-sided average to make it balanced.  This yields a well-defined knot invariant.  However, it is not a knot concordance invariant.  In particular, it does not vanish for slice knots (knots that bound smooth embedded disks in $B^4$); specifically, there are slice knots having non-zero (unbalanced) signature functions~\cite{cha-livingston, levine3}.  In contrast to this, the (balanced) signature function induces a well-defined homomorphism from the knot concordance group to the set of functions on $\ss$.

\vskip.05in
\noindent{\it Acknowledgments.} Thanks are due to Jae Choon Cha for pointing out the result presented in Lemma~\ref{lem:irred}. 
 
 \section{Definition of the signature function}
 
 The signature function of a knot is defined in terms of the  Seifert matrix of the knot, $V_K$.  Associated to $V_K$ there  is the   matrix $$W_K(t) = (1-t)V_K + (1-t^{-1})V_K^{\sf T}$$ with entries in the field of fractions,  $\qq(t)$. This matrix is hermitian with respect to the involution of $\qq(t)$ induced by  $t \to t^{-1}$.  Substituting $\omega \in \ss$ for $t$ yields a complex hermitian matrix, having signature which we temporarily denote $s_K(\omega)$.  Then $\sigma_K$ is defined by
$$\sigma_K(e^{2 \pi i x}) = \frac{1}{2}\lim_{\epsilon\to 0^+} \left( s_K(e^{  2 \pi i (x+\epsilon)}) + s_K(e^{ 2 \pi i (x-\epsilon)} )\right).$$

For almost all $\omega \in \ss$, $W_K(\omega)$  nonsingular.  For these $\omega$,  $s_K(\omega) \equiv \text{rank}(W_K(\omega))\equiv 0 \mod 2$.  If follows that $\sigma_K(\omega)$ is an integer for all $\omega$.  Similarly, the jump function $J_K$ takes on integer values. 
 
\section{Proof of necessity}
The necessity of the first three conditions is well-known, with many references.  The fourth condition is also known, but is not stated explicitly in the literature.  Summary proofs are included for completeness.\vskip.05in
 
\noindent{\bf Property 1:} The necessity of Property (1) follows from the fact that a hermitian matrix and its complex conjugate have the same signature.\vskip.05in

\noindent{\bf Property 2:} If $\omega$ close to 0 with positive argument, we can use a Taylor approximation to write $\omega = 1 +   \nu i - \nu^2 g(\nu)$ where $\nu \in \rr_+$ is close to 0 and $g$ is a real-valued  differentiable function defined near 0.  In terms of $\nu$, 
$$W_K(\omega) = \nu i (-V_K + V_K^{\sf T}) + \nu^2g(\nu)(V_K +V_K^{\sf T}).$$
The signature of this matrix is the same as that of 
$$  i (-V_K + V_K^{\sf T}) - \nu g(\nu)(V_K +V_K^{\sf T}).$$
For a knot, the matrix $  i (-V_K + V_K^{\sf T}) $ is congruent to the direct sum of $2 \times 2$ matrices, each of the form

$$ \left( \begin{array}{cc}
0& i \\
-i & 0 \\
\end{array} \right).$$ 
(This is false for links.)  This is nonsingular with signature 0.  A small perturbation leaves the signature unchanged.\vskip.05in
    
\noindent{\bf Property 3:}    We can rewrite 
 $$W_K  = (1-t) (V_K - t^{-1}V_K^{\sf T}).$$
 This matrix is nonsingular except at roots of $\det(V_K - t^{-1}V_K^{\sf T})$ and  at $t = 1$, and hence the signature  is locally constant away from such roots.  We have just seen that $t=1$   is not a singular point: $\sigma_K(\omega) = 0$ for $\omega$ near 1.  The Alexander polynomial can be defined as $\Delta_K(t) = \det(V_K - tV_K^{\sf T})$.  Replacing $t$ with $t^{-1}$ does not change this determinant, modulo multiplication by $\pm t^k$ for some $k \in\zz$.  Thus, all singularities of the signature function occur at roots of $\Delta_K(t)$.\vskip.05in

\noindent{\bf Property 4:} Since $\qq(t)$ is a field of characteristic 0, the form $W_K$ can be diagonalized.  Thus, we can prove the necessity of Property (4) by verifying it for $1 \times 1$ forms and applying the additivity of signature.  In general, a $1 \times 1$ matrix  is given as $(h(t))$ for some rational function $h(t)$ that is invariant under the involution $t \to t^{-1}$.  A change of basis permits us to clear the denominator and all factors  of the form $g(t)g(t^{-1})$.  That is, we can assume $h(t)$ is a product of distinct irreducible symmetric polynomials.

Suppose that $\alpha_i$ and $\alpha_j$ are roots of the irreducible polynomial $\delta(t)$. After a change of basis,  the matrix is of the form $\left( f(t)\delta(t)^\epsilon \right)$, with $\epsilon = 0$ or $\epsilon =1$, and with $f(t)$ symmetric and relatively prime to $\delta(t)$.  In the first case,  $\epsilon =0$,  there is no jump at either $\alpha_i$ and the signature is $\pm 1$.   In the second case, the jumps are either $\pm 1$.  Thus, for $W_K$ the  jumps are equal mod 2.  


\section{Proof of sufficiency}

\subsection{Background}
The proof of sufficiency depends on some previously known facts which we collect here as a series of lemmas.

\begin{lemma}\label{lem:irred}  If $\alpha \in \ss$ is the root of an Alexander polynomial $\Delta(t)$, then it is the root of an irreducible Alexander polynomial.
\end{lemma}
\begin{proof}  Suppose that $\delta_1(\alpha) = 0$, where $\delta_1(t)$ is an irreducible integral factor of $\Delta(t)$.  Clearly $\delta_1(1) = \pm 1$.  It remains to prove the symmetry of $\delta_1(t)$.  

Normalize $\delta_1(t)$ so that $\delta_1(t) \in \zz[t]$ with nonzero constant coefficient. Let $k$ denote the degree of $\delta_1(t)$ and let $\delta_2(t) = t^k\delta_1(t^{-1})$. We have $\delta_2(\alpha^{-1}) = 0$.  But $\alpha^{-1} = \overline{\alpha}$, so $\delta_2(\overline{\alpha}) = 0$.  On the other hand,   $\delta_1$ has real coefficients, so $\delta_1(\overline{\alpha}) = 0$.  Thus, $\delta_1(t) = a \delta_2(t)$ for some $a\in \qq$; that is, $\delta_1(t) = at^k \delta_1(t^{-1})$.  Letting $t = 1$, we have $\delta_1(1) = a\delta_1(1)$.  Since   $\delta_1(1) \ne 0$, we have $a = 1$, implying the symmetry of $\delta_1(t)$.
\end{proof}

\begin{lemma}\label{lem:jump} If $\omega \in \ss$ is the root of a symmetric irreducible polynomial $\delta(t)$ that has odd exponent as a factor of $\Delta_K(t)$, then $J_K(\omega) \equiv 1 \mod 2$.
\end{lemma}
\begin{proof} The diagonalization process used in the proof of the necessity of Property (4)   does not change the exponent of $\delta(t)$ as a factor of the determinant, modulo 2.  Thus, after diagonalization there are an odd number of diagonal entries of the form $f(t)\delta(t)$, each of which contributes $\pm 1$ to the jump function.
\end{proof}

\begin{lemma}\label{lem:poly1} For every Alexander polynomial $\Delta(t)$, there is an unknotting number one knot $K$ with $\Delta_K(t) = \Delta(t)$.
\end{lemma}
\begin{proof} This is a theorem of Kondo~\cite{kondo}.
\end{proof}

\begin{lemma}\label{lem:dense} There is a dense subset $\{\beta_i\}$ of  $\ss$, each of which is the root of a quartic Alexander polynomial having precisely two  roots on $\ss$. Consequently, for every $\omega \in \ss$, there is a $\omega'$ arbitrarily close to $\omega$ and a knot $K$ such that $\sigma_K$ has jumps only at $\omega'$ and $\overline{\omega'}$.
\end{lemma}
\begin{proof} This is proved in~\cite{cha-livingston}\end{proof}

\begin{lemma}\label{lem:genus}If $K$ can be unknotted by changing one crossing from negative to positive, then $0\le \sigma_K(\omega) \le 2$ for all $\omega \in \ss$.
\end{lemma}
\begin{proof} A proof of this for the classical signature, $\sigma_K(-1)$, appears in~\cite{cochran-lickorish}.  That proof can be generalized by considering $p$--fold covers rather than $2$--fold covers.  Here is a brief argument, similar to one given in~\cite{livingston-tau} for the Heegaard Floer $\tau$--invariant.

Fix the value of $\omega$.   Since $K$ can be unknotted with one crossing change, it bounds a punctured torus embedded in $B^4$, and thus $\left|\sigma_K(\omega)\right| \le 2$ (see for instance,~\cite{taylor}).   By the previous lemmas, there is a knot $J$ that can be unknotted with a single crossing change from positive to negative having $\sigma_J(\omega) = -2$.   The knot $K \cs J$ can be unknotted with two crossing changes of opposite sign.  Thus, it bounds a disk in $B^4$ with two double points  having opposite sign.  This surface can be surgered  to eliminate both double points using the neighborhood of an arc on the surface joining the two double points.  This surgery yields a punctured torus bounded by $K \cs J$.
 It follows that $\left|  \sigma_K(\omega) + \sigma_J(\omega)\right| \le 2$.  The rest is arithmetic.
\end{proof}

\subsection{Proof of sufficiency}
The following simple observation, though not needed in the proof, might clarify the  conditions on $\sigma$ given in the statement of Theorem~\ref{thm:main}.  Any step function $\sigma$ satisfying the stated properties is necessarily even-valued away from its discontinuities and for all $\alpha \in \ss$, $J(\alpha) \equiv \sigma(\alpha) \mod 2$.

Since the signature and jump functions are additive under connected sums of knots, we can focus on signature functions whose jumps occur at the roots of a single irreducible Alexander polynomial.  Let $\delta_1(t)$ be an irreducible Alexander polynomial having roots $\{\alpha_1, \ldots, \alpha_k\}$ on the {\it upper} unit circle, $\ss_+$.  By Lemmas~\ref{lem:jump}~and~\ref{lem:poly1}, there is a knot $K$ whose signature function has an odd jump at each $\alpha_i$.  With these observations, the proof of sufficiency is reduced to the following theorem.

\begin{theorem} For each $\alpha_m$, there exists a knot $K$ with jump function satisfying $J_K(\alpha_i) = 2$ and $J_K(\omega) = 0$ if $\omega \ne \alpha_i$.
\end{theorem}

\begin{proof}  Assume the set of numbers $\{\alpha_1, \ldots, \alpha_k\}$ is ordered by increasing argument.  We focus on one element of the set, $\alpha_m$.  Choose a $\beta$ with argument between that of $\alpha_{m}$ and $\alpha_{m+1}$.  (In the case $m=k$, choose $\beta\in\ss_+$ with argument greater than that of $\alpha_k$.)  Furthermore, according to Lemma~\ref{lem:dense} we can assume $\beta $ is the root of an irreducible Alexander polynomial, $\delta_2(t)$, having a unique root on the upper unit circle $\ss_+$.   

Let $K_1$, $K_2$, and $K_3$ be unknotting number one knots having Alexander polynomials $\delta_1(t)$, $\delta_2(t)$, and $\delta_1(t)\delta_2(t)$, repectively.  By changing orientation if necessary, we can assume that each is unknotted by changing a negative crossing to positive.  The signature functions for these knots can have nontrivial jumps only at elements of the set $\{\alpha_1, \alpha_2, \ldots, \alpha_{m}, \beta, \alpha_{m+1}, \ldots, \alpha_k\}$.  By Lemmas~\ref{lem:jump}~and~\ref{lem:genus}, the jump functions of each must be as follows.  (We write the jump at $\beta$ in bold to highlight its location in the list in position $(m+1)$; the entire list is an ordered $(k+1)$--tuple.)

\begin{itemize}
\item Jumps for $\sigma_{K_1}$: $[ 1, -1, \ldots, (-1)^{m+1}, {\bf 0}, (-1)^{m}, \ldots, (-1)^{k+1}]$.\vskip.05in
\item Jumps for $\sigma_{K_2}$: $[ 0, 0, \ldots, 0, {\bf 1},0, \ldots, 0]$.\vskip.05in
\item Jumps for $\sigma_{K_3}$: $[ 1, -1, \ldots, (-1)^{m+1}, {\bf (-1)^{m}}, (-1)^{m+1}, \ldots, (-1)^{k}]$.\vskip.05in
\end{itemize}

We now see that the jumps for the connected sum $J_m =K_1 \cs - (-1)^m K_2 \cs K_3$ are given by
$$[ 2, 2, \ldots, 2, {\bf 0},0, \ldots, 0].$$
Since the jumps for this knot occur only at $\alpha_i$, we list the jumps at those points as a $k$--tuple:  $$[2,2,\ldots,2,0, \ldots, 0].$$  The last nonzero entry is in the $m$  position.

This entire construction can be repeated with $m$ incremented by $1$, building a knot $J_{m+1}$.  The knot $J_m \cs -J_{m+1}$ is our desired knot.
\end{proof}




\begin{thebibliography}{BBVB}

\bibitem{cha-livingston} J.~C.~Cha and C.~Livingston, {\em Knot signature functions are independent}, Proc. Amer. Math. Soc. {\bf 132} (2004),  2809--2816.

\bibitem{cochran-lickorish} T.~Cochran and W.~B.~R.~Lickorish, {\em  Unknotting information from 4-manifolds}  Trans. Amer. Math. Soc. {\bf 297} (1986)  125--142.

\bibitem{kondo}  H.~Kondo, {\em Knots of unknotting number 1 and their Alexander polynomials}, Osaka J. Math. {\bf 16} (1979), 551--559.

\bibitem{levine2} J.~Levine, {\em A characterization of knot polynomials}, Topology {\bf 4 } (1965) 135--141.

\bibitem{levine}  J.~Levine,  {\em Invariants of knot cobordism}, Invent.~Math. {\bf 8} (1969), 98--110.

\bibitem{levine3} J.~Levine, {\em Metabolic and hyperbolic forms from knot theory}, J. Pure Appl. Algebra {\bf 58} 
(1989), 251--260.

\bibitem{livingston-tau} C.~Livingston, {\em Computations of   the Ozsv\'ath-Szab\'o knot concordance invariant}, Geom. Topol. {\bf 8} (2004) 735--742.

\bibitem{matumoto} Y.~Matumoto, {\em On the signature invariants of a non-singular complex sesquilinear form,} J. Math. Soc. Japan {\bf 29} (1977) 67--71.

\bibitem{milnor}  J.~ Milnor, {\em Infinite cyclic coverings}, Conference on the Topology of Manifolds (Michigan State Univ., E. Lansing, Mich., 1967) Prindle, Weber \& Schmidt, Boston, Mass., (1968)  11--133. 

\bibitem{murasugi}  K.~Murasugi, {\em  On a certain numerical invariant of link types}, Trans. Amer. Math. Soc. {\bf 117} (1965) 387--422. 

\bibitem{seifert} H.~Seifert, {\em  \"Uber das Geschlecht von Knoten}, Math. Ann. {\bf 110}   571--592.

\bibitem{taylor} L.~Taylor, {\em  On the genera of knots},  Topology of low-dimensional manifolds (Proc. Second Sussex Conf., Chelwood Gate, 1977),  144--154, Lecture Notes in Math., {\bf 722} , Springer, Berlin, 1979.

\bibitem{tristram} A. Tristram, {\em Some cobordism invariants for links},   Proc. Camb. Phil. Soc. {\bf 66} (1969), 251--264.

\bibitem{trotter} H.~Trotter, {\em Homology of groups systems with applications to knot theory}, Ann. of Math. {\bf 76} (1962) 464--498.

\end{thebibliography}
\end{document}